\documentclass[graybox]{svmult}

\usepackage{type1cm}        % activate if the above 3 fonts are
\usepackage{makeidx}         % allows index generation
\usepackage{graphicx}        % standard LaTeX graphics tool
\usepackage{multicol}        % used for the two-column index
\usepackage[bottom]{footmisc}% places footnotes at page bottom

\usepackage{newtxtext}       %
\usepackage{newtxmath}       % selects Times Roman as basic font

\usepackage[english]{babel}
\usepackage[babel=true]{microtype}
\usepackage{amsmath}

\usepackage{algorithm}
\usepackage{algpseudocode}

\usepackage{siunitx}
\usepackage{hyperref}
\usepackage{cleveref}
\crefname{equation}{}{}
\Crefname{equation}{Equation}{Equations}
\crefname{figure}{Fig.}{Figs.}
\Crefname{figure}{Figure}{Figures}
\crefname{table}{Table}{Tables}
\crefname{algorithm}{Algorithm}{Algorithms}
\crefname{section}{Sect.}{Sects.}
\Crefname{section}{Section}{Sections}

\crefname{theorem}{Theorem}{Theorems}
\Crefname{theorem}{Theorem}{Theorems}
\crefname{corollary}{Corollary}{Corollaries}
\Crefname{corollary}{Corollary}{Corollaries}
\crefname{lemma}{Lemma}{Lemmas}
\Crefname{lemma}{Lemma}{Lemmas}
\crefname{definition}{Definition}{Definitions}
\Crefname{definition}{Definition}{Definitions}
\crefname{remark}{Remark}{Remarks}
\Crefname{remark}{Remark}{Remarks}

\def\relint{\operatorname{relint}}
\def\conv{\operatorname{conv}}\newcommand{\bx}{\mathbf{x}}
\newcommand{\mt}{{\mathcal T}}

\makeindex             % used for the subject index
\begin{document}

\title*{Manifolds of triangulations, braid groups of manifolds, and the groups $\Gamma_{n}^{k}$}
\author{Denis A. Fedoseev, Vassily O. Manturov and Igor M. Nikonov}
\institute{%
   Denis A. Fedoseev \at{} Moscow State University, Moscow Center for Fundamental and Applied Mathematics, \email{denfedex@yandex.ru}
   \and{}
   Vassily O. Manturov \at{} Moscow Institute of Physics and Technology, Russia; North-Eastern University, China; Kazan Federal University, Russia, \email{vomanturov@yandex.ru}
   \and{}
   Igor M. Nikonov \at{} Moscow State University, \email{nim@mail.ru}
}
\maketitle

\abstract*{The spaces of triangulations of a given manifold have been widely studied. The celebrated theorem of Pachner~\cite{Pachner} says that any two triangulations of a given manifold can be connected by a sequence of bistellar moves, or Pachner moves, see also~\cite{GKZ,Nabutovsky}. In the present paper we consider groups which naturally appear when considering the set of triangulations with fixed number of simplices of maximal dimension. There are three ways of introducing this groups: the geometrical one, which depends on the metric, the topological one, and the combinatorial one. The second one can be thought of as a ``braid group'' of the manifold and, by definition, is an invariant of the topological type of manifold; in a similar way, one can construct the smooth version. We construct a series of groups $\Gamma_{n}^{k}$ corresponding to Pachner moves of $(k-2)$-dimensional manifolds and construct a canonical map from the braid group of any $k$-dimensional manifold to $\Gamma_{n}^{k}$ thus getting topological/smooth invariants of these manifolds.}

\abstract{The spaces of triangulations of a given manifold have been widely studied. The celebrated theorem of Pachner~\cite{Pachner} says that any two triangulations of a given manifold can be connected by a sequence of bistellar moves, or Pachner moves, see also~\cite{GKZ,Nabutovsky}. In the present paper we consider groups which naturally appear when considering the set of triangulations with fixed number of simplices of maximal dimension. There are three ways of introducing this groups: the geometrical one, which depends on the metric, the topological one, and the combinatorial one. The second one can be thought of as a ``braid group'' of the manifold and, by definition, is an invariant of the topological type of manifold; in a similar way, one can construct the smooth version. We construct a series of groups $\Gamma_{n}^{k}$ corresponding to Pachner moves of $(k-2)$-dimensional manifolds and construct a canonical map from the braid group of any $k$-dimensional manifold to $\Gamma_{n}^{k}$ thus getting topological/smooth invariants of these manifolds.}

%---------------------------------------------------------------------
\section*{Introduction}

Returning back to 1954, J.W.~Milnor in his first paper on link groups~\cite{Milnor} formulated the idea that manifolds which share lots of well known invariants (homotopy type, etc.) may have different {\em link groups}.

In 2015 the second named author~\cite{MN} defined a series of groups $G_{n}^{k}$ for natural numbers $n>k$ and formulated the following principle:

\begin{important}{The main principle}
If dynamical systems describing the motion of $n$ particles possess a nice codimension one property depending exactly on $k$ particles, then these dynamical systems admit a topological invariant valued in $G_{n}^{k}$.
\end{important}

We shall define braids on a manifold as loops in some configuration space related to the manifold.
Following $G^k_n$-ideology, we mark singular configurations, while moving along a loop in the configuration space, in order to get a word  --- an element of a ``$G^k_n$-like'' group $\Gamma_n^k$. Singular configurations will arise here from the moments of transformation of triangulations spanned by the configuration points; and in this case  the good property of codimension one is {\em ``$k$ points of the configuration lie on a sphere of dimension $k-3$ and there are no points inside the sphere''}.

In the present paper, we shall consider an arbitrary closed $d$-manifold $M^{d}$ and for $n$ large enough we shall construct the braid group $B_{n}(M^{d})=\pi_1(C_n(M^d))$ (see Definitions~\cref{def:geometric_n_braid}, \cref{def:topological_n_braid}, \cref{def:combinatorial_n_braid}), where $C_n(M^d)$ is the space of generic (in some sense) configurations of points in $M^d$ which are dense enough to span a triangulation of $M$. This braid group will have a natural map to the group $\Gamma_{n}^{d+2}$. %Here $n$ will be estimated in terms of minimal number of triangles of a triangulation of $M$.
We shall consider point configurations for which the Delaunay triangulation exists, see~\cite{BDGM}. \\

The idea of this work comes from the second named author, and the third named author defined the constructions of $\Gamma_n^k$, $k\ge 5$, given in \cref{sect:Gamma_nk_construction}. \\

\begin{acknowledgement}
The authors are very grateful to S.~Kim for extremely useful discussions and comments.

The first named author was supported by the Russian Foundation for Basic Research (grant No.~19-01-00775-a, grant No.~20-51-53022). The second named author was supported by the Russian Foundation for Basic Research (grant No.~20-51-53022, grant No.~19-51-51004). The third named author was supported by the Russian Foundation for Basic Research (grant No.~18-01-00398-a, grant No.~19-51-51004).
The work of V.O.M. was also funded by the development program of the Regional Scientific and
Educational Mathematical Center of the Volga Federal District, agreement No.~075-02-2020.
\end{acknowledgement}

\section{The Manifold of Triangulations}
\label{sec:manifold_of_triangulations}
In the present section we define three types of manifold of triangulations: a geometrical one, a topological one, and a combinatorial one. We begin with geometrical approach, considering a manifold with a Riemannian metric.

\subsection{Geometrical manifold of triangulations}
\label{sec:geometrical_manifold}

Fix a smooth manifold $M$ of dimension $d$ with a Riemannian metric $g$ on it. Let $n\gg d$ be a large natural number. Consider the set of all Delaunay triangulations of $M$ with $n$ vertices $x_{1},\dots, x_{n}$ (if such triangulations exist). Such Delaunay triangulations are indexed by sets of vertices% $x_{1},\dots, x_{n}$
, hence, the set of such triangulation forms a subset of the configuration space, which we denote by $C(M,n)$. The above subset will be an open (not necessarily connected) manifold of dimension $dn$: if some $n$ pairwise distinct points $(x_{1},\dots, x_{n})$ form a Delaunay triangulation then the same is true for any collection of points $(x'_{1},\dots, x'_{n})$ in the neighbourhood of $(x_{1},\dots, x_{n})$. We call a set of points $x_{1}, \dots, x_{n}$ {\em admissible} if such a Delaunay triangulation exists.

Hence, we get a non-compact (open) manifold $M_{g}^{dn}$. %This manifold has a natural structure. Namely, we 
We call two sets of points $(x_{1},\dots, x_{n})$ and $(x'_{1},\dots, x'_{n})$ {\em adjacent} if there is a path $(x_{1,t},\dots, x_{n,t})$ for $t\in [0,1]$ such that $(x_{1,0},\dots,x_{n,0})=(x_{1},\cdots, x_{n})$ and $(x_{1,1},\dots, x_{n,1})=(x'_{1},\dots, x'_{n})$ and there exists exactly one $t_{0}\in [0,1]$ such that when passing through $t=t_{0}$ the Delaunay triangulation $(x_{1,t},\dots, x_{n,t})$ undergoes a flip (also called {\em Pachner moves} or {\em bistellar moves}). %Generally, the Pachner moves, which do not change the number of points, may be not enough. In that case, we shall have many components and the invariant (with values in $\Gamma_n^k$) will be a set of groups. In dimension $3$ the Pachner moves are sufficient~\cite{Matveev}.

The flip (Pachner move) corresponds to a position when some $d+2$ points $(x_{{i_1},t_{0}},\dots, x_{{i_{d+2}},t_{0}})$ belong to the same $(d-1)$-sphere $S^{d-1}$ such that no other point $x_{j}$ lies inside the ball $B$ bounded by this sphere.

%The Figures for flips (also called {\em Pachner moves} or {\em bistellar moves}).

%Note that the above definition heavily depends on the metric of the manifold $M$. For example, if we take $M$ to be the torus glued from the square $1\times 10$, the combinatorial structure of the manifold of triangulations will differ from the combinatorial structure for the case of the manifold of triangulations for the case of the torus glued from the square $1\times 1$. The {\em topological manifold of triangulations} which is independent on the metric will be defined in the next section.

Every manifold of triangulations described above has a natural stratification. Namely, every point of $M_{g}^{dn}$ is given by a collection $(x_{1},\dots, x_{n})$. Such a collection is  {\em generic} if there is no sphere $S^{d-1}$ such that exactly $d+2$ points among $x_{k}$ belong to this sphere without any points inside the sphere.

We say that a point $(x_{1},\dots, x_{n})$ is {\em of codimension $1$}, if there exists exactly one sphere with exactly $d+2$ points on it and no points inside it; analogously, {\em codimension $2$ strata} correspond to either one sphere with $d+3$ points or two spheres containing $d+2$ points each.

In codimension $2$ this corresponds to either one point of valency five in Vorono\"{\i} tiling or two points of valency four in Vorono\"{\i} tiling.

Hence, we have constructed a stratified (open) manifold $M_{g}^{dn}$. We call it the {\em geometrical manifold of triangulations}\index{triangulation manifold!metrical}.

Note that the manifold $M_{g}^{dn}$ may be not connected, i.e. there can exist non-equivalent triangulations. On the other hand, if one considers the spines of some manifold, they all can be transformed into each other by Matveev--Piergallini moves~\cite{Matveev}.

%%%%%%%%%%%%%%%%%%%%%%%%%%%%%%%
Denote the connected components of $M_{g}^{dn}$ by $(M_{g}^{dn})_{1},\dots, (M_{g}^{dn})_{p}$.

\begin{definition}
The {\em geometrical $n$-strand braid groups} of the manifold $M_{g}$ are the fundamental groups
$$B_g(M_{g},n)_{j} = \pi_{1}((M_{g}^{dn})_{j}), \quad j=1,\dots, p.$$
\label{def:geometric_n_braid}
\end{definition}
%%%%%%%%%%%%%%%%%%%%%%%%%%%%%%%

\subsection{Topological manifold of triangulations}
\label{sec:topological_manifold}

The definition of geometrical manifold of triangulations heavily depends on the metric of the manifold $M$. For example, if we take $M$ to be the torus glued from the square $1\times 10$, the combinatorial structure of the manifold of triangulations will differ from the combinatorial structure for the case of the manifold of triangulations for the case of the torus glued from the square $1\times 1$. Its analogue which is independent on the metric is the {\em topological manifold of triangulations}. We shall construct the $2$-frame of this manifold. The idea is to catch all simplicial decompositions which may eventually happen for the manifold of triangulations for whatever metrics, and glue them together to get an open (non-compact) manifold.

Consider a topological manifold $M^{d}$. We consider all Riemannian metrics $g_{\alpha}$ on this manifold. They give manifolds $M_{g_{\alpha}}^{dn}$ as described in \cref{sec:geometrical_manifold}, which are naturally stratified. By a {\em generalised cell} of such a stratification we mean a connected component of the set of generic points of $M_{g_{\alpha}}^{dn}$.

We say that two generalised cells $C_{1}$ and $C_{2}$ are {\em adjacent} it there exist two points, say, $x=(x_{1},\dots, x_{n})$ in $C_{1}$ and $x'=(x'_{1},\dots, x'_{n})$ in $C_{2}$, and a path $x_{t}=(x_{1}(t),\dots, x_{n}(t))$, such that $x_{i}(0)=x_{i}$ and $x_{i}(1)=x'(i)$ such that all points on this path are generic except for exactly one point, say, corresponding to $t=t_{0}$, which belongs to the stratum of codimension $1$.

We say that two generic strata of $M_{g_{\alpha}}^{dn}, M_{g_{\beta}}^{dn}$ are {\em equivalent} if there is a homeomorphism of $M_{g_{\alpha}}^{dn}\to M_{g_{\beta}}^{dn}$ taking one stratum to the other.

A generalised {\em $0$-cell} of the manifold of triangulations is an equivalence class of generic strata.

Analogously, we define generalised {\em $1$-cells} of the manifold of triangulations as equivalence
classes of pairs of adjacent vertices for different metrics $M_{g_{\alpha}}^{dn}$.

In a similar manner, we define generalised {\em $2$-cells} as equivalence classes of discs for
metrics $M_{g_{\alpha}}^{dn}$ such that:
\begin{enumerate}
\item vertices of the disc are points in $0$-strata;
\item edges of the disc connect vertices from adjacent $0$-strata;
each edge intersects codimension $1$ set exactly in one point;
\item the cycle is spanned by a disc which intersects codimension $2$ set exactly
at one point;
\item equivalence is defined by homeomorphism taking disc to disc,
edge to an equivalent edge and vertex to an equivalent vertex and
respects the stratification.
\end{enumerate}

Thus, we get the $2$-frame of the manifold $M^{dn}_{top}$. This manifold might be disconnected.

\begin{definition}
The {\em topological $n$-strand braid groups} of the manifold $M$ are the fundamental groups
$$B_t(M,n)_{j} = \pi_{1}((M^{dn}_{top})_{j}), \quad j=1,\dots, q.$$
\label{def:topological_n_braid}
\end{definition}

%Now, we define the {\em topological braid groups} of $B(M,n)_{j}, j=1,\dots, q$, as
%fundamental groups $\{\pi_{1}((M^{dn}_{top})_{j})\},
%j=1,\dots, q$.
%%%%%%%%%%%%%%%%%%%%%%%%%%%%%

%\subsection{The topological definition}

%Let $g$ and $g'$ be some two metrics on the manifold $M$. We say that the manifolds of triangulations $M_{g}^{dn}$ and $M_{g'}^{dn}$ are {\em combinatorially equivalent} if there is a homeomorphism between  $M_{g}^{dn}$ %$T^{n}(M^{n},g)$
%and $M_{g'}^{dn}$ which preserves the stratification. %We call the resulting manifold $T^{n}(M^{n},g)$ the {\em metrical manifold of triangulations}\index{triangulation manifold!metrical}.
%
%
%Let us list all generalised cell types for $M$ for all possible metrics $(M,g)$. Now, if two generalised cell types $C,C'$ are adjacent for some metrics $g$, we shall paste them in the same way they are pasted in $g$.

%{\bf Draw a picture : Codimension 1 corresponds to a flip; denoted by a point}
%~\cref{fig:delaunay_flip}
\begin{figure}
\centering\includegraphics[width=200pt]{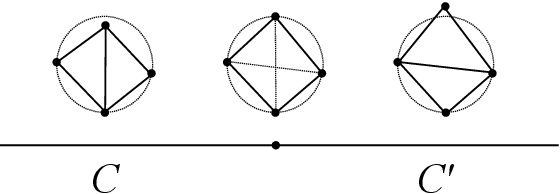}
\caption{Codimension 1 corresponds to a flip}\label{fig:delaunay_flip}
\end{figure}

%If for some metrics $g$ there is a more complicated stratum, we s

%{\bf Draw a picture: Codimension 2 corresponds to five nearest points on a circle; codimension 1 strata correspond to five possible flips for quadruples of points; thus, we have five rays emanating from a vertex.}
%~\cref{fig:delaunay_pentagon}

\begin{figure}[ht]
\sidecaption{}
\includegraphics[width=180pt]{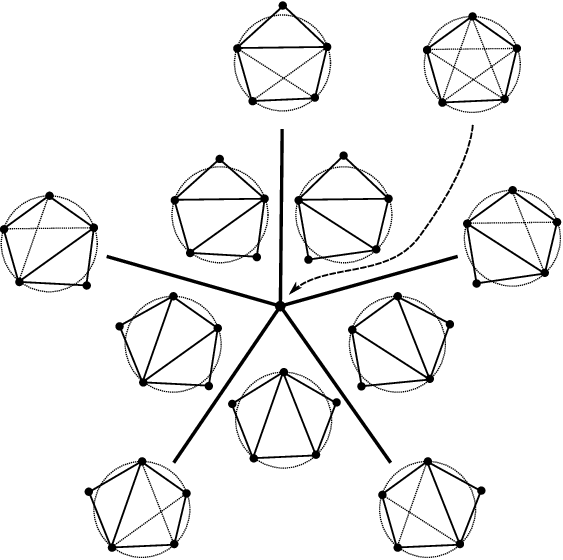}
\caption{Codimension 2 corresponds to a pentagon}\label{fig:delaunay_pentagon}
\end{figure}

%Hence, we take all possible generalised cell types and glue them together as if they corresponded to some concrete metrics.

%We call the resulting manifold {\em the manifold of triangulations} of $M$ with $n$ points; notation: $T^{n}(M^{d})$. \index{triangulation manifold!topological}

%Note that if $n$ is not large enough (so that there are no triangulations with $n$ points of the manifold $M$), then manifolds $T^{n}(M^{d},g)$ and $T^{n}(M^{d})$.

%{\bf An example is needed.}

\subsection{Combinatorial manifold of triangulations}

In \cref{sec:geometrical_manifold} and \cref{sec:topological_manifold} we constructed manifolds of triangulations basing on the notion of {\em Delaunay triangulation}. We did not actually discuss the {\em existence} of such triangulation for a given manifold and a given set of points, requesting only that the number of points is sufficiently large and points are sufficiently dense. It turns out that this condition is not always sufficient, as shows in \cite{BDGM}. Even for a large number of points and a Riemannian manifold with metric arbitrary close to Euclidian one, there may not exist a Delaunay triangulation.

One may impose additional restrictions on the vertex set or on the manifold itself to overcome this difficulty, but to work with the most general situation we shall consider the third notion of manifolds of triangulations: the {\em combinatorial manifold of triangulations} which we denote by $M^{dn}_{comb}$. We construct the 2-frame of the manifold $M^{dn}_{comb}$ which is needed to get the fundamental group.

First, we fix $n$ points on the manifold $M^d$ and consider triangulations of the manifold with vertices in those points. Now we do not restrict ourselves to Delaunay triangulations, but consider all triangulations of the manifold. Moreover, we work with {\em equivalence classes} of triangulations modulo the following relation: two triangulations $T_1, T_2$ with (fixed) vertices $v_1,\dots, v_n$ are said to be {\em equivalent} if and only if for each pair $(i,j)\in\bar{n}$ the vertices $v_i, v_j$ are connected by an edge of the triangulation $T_1$ if and only if they are connected by an edge of the triangulation $T_2$. From now on saying ``triangulation'' we mean such equivalence class of triangulations. Those triangulations are the vertices (0-cells) of the frame of the manifold $M^{dn}_{comb}$. 

Two triangulations are connected by an edge (a 1-cell of the frame) if and only if they differ by a {\em flip} --- a Pachner move.

Finally, the 2-cells of the frame are chosen to correspond to the relations of the groups  $\Gamma_n^k$ (see \cref{sect:Gamma_nk_construction}). There are two types of such relations. One relation corresponds to the far commutativity, i.e., for any two {\em independent} flips $\alpha,\beta$ (flips related to different edges) we define a quadrilateral consisting of subsequent flips corresponding to $\alpha,\beta,\alpha^{-1},\beta^{-1}$. The other relations correspond to all possible simplicial polytopes with $d+3$ vertices inscribed in the unit sphere. With each polytope of such sort, we associate a relation of length $d+3$ as shown in \cref{sect:Gamma_nk_construction}.

\begin{definition}
The {\em combinatorial $n$-strand braid groups} of the manifold $M$ are the fundamental groups
$$B_c(M,n)_{j} = \pi_{1}((M^{dn}_{comb})_{j}), \quad j=1,\dots, q.$$
\label{def:combinatorial_n_braid}
\end{definition}

\section{The groups $\Gamma^k_n$}\label{sect:Gamma_nk_construction}

In the present section we consider the case of the manifold $M^d=\mathbb{R}^d$ being a vector space. The case of Euclidian space gives groups which are important for the general case. In the present paper we deal mostly with the case of $d\ge 3$. The very interesting initial case $d=2$ is studied in~\cite{ManturovKim}. We start with the case $d=3$.

\subsection{The case $d=3$}
Consider a configuration of $n$ points in general position in $\mathbb{R}^3$. We can think of these points as lying in a fixed tetrahedron $ABCD$. The points induce a unique Delaunay triangulation of the tetrahedron: four points form a simplex of the triangulation if and only if there is no other points inside the sphere circumscribed over these points. The triangulation transforms when the points move in the space.

In order to avoid degenerate Delaunay triangulations we exclude configurations where four points lie on one circle (intersection of a plane and a sphere).

Transformations of the combinatorial structure of the Delaunay triangulation correspond to configurations of codimension $1$ when five points lie on a sphere which does not contain any points inside. At this moment two simplices of the triangulation are replaced with three simplices as shown in \cref{fig:pachner_move} (or vice versa). This transformation is called a {\em 2-3 Pachner move}\index{Pachner move}. %It is known that any two triangulation of a convex polyhedron with the same number of vertices can be connected by a sequence of Pachner moves.

To trace the evolution of triangulations that corresponds to a dynamics of the points we attribute a generator to each Pachner move. For the move that replace the simplices $iklm, jklm$ with the simplices $ijkl, ijkm, ijlm$ in \cref{fig:pachner_move} we use the generator $a_{ij,klm}$. Note that 1) we can split the indices into two subsets according to the combinatorics of the transformation; 2) the generator $a_{ij,klm}$ is not expected to be involutive because it changes the number of simplices of the triangulation.

%\begin{remark}
%	Note that \cref{fig:pachner_move} depicts not only the Pachner move, but the dual Matveev-Piergallini move as well.
%\end{remark}

\begin{figure}
\centering\includegraphics[width=300pt]{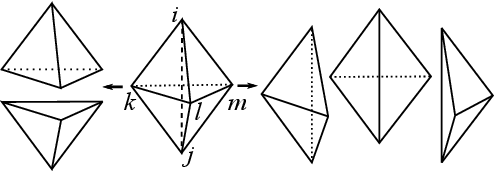}
\caption{A Pachner move}
\label{fig:pachner_move}
\end{figure}

The relations on generators $a_{ij,klm}$ correspond to configurations of codimension $2$ which occurs when either 1) six points lie on the same sphere with empty interior, or 2) there are two spheres with five points on each of them, or 3) five points on one sphere compose a codimension $1$ configuration.

The last case means that the convex hull of the five points has a quadrilateral face (\cref{fig:delauney_pentahedron}). The vertices of this face lie on one circle so we exclude this configuration.
%This face consists of two triangular faces, say $kln, lmn$. Let $jknl$ and $j'lmn$ be the simplices adjacent to the pyramid. If $j$ and $j'$ do not coincide we have three circumscribed spheres which have a common circle (it lies in the plane of the face $klmn$). But the union of two spheres among them contains the third one. On the other hand, the interior of each sphere must be empty. Thus, $j$ and $j'$ must coinside, and we have a quadrilateral bipyramid $ijklmn$. There are two triangulation of it: $\{ikln, ilmn, jkln, jlmn\}$ and $\{iklm, ikmn, jklm, jkmn\}$, which differ by a suspension of two dimensional flip. The suspension can be decomposed into a sequence of two Pachner moves:
%$$a^{}_{km,iln} a^{-1}_{ln,jkm}=a^{}_{km,jln}a^{-1}_{ln,ikm}.$$

\begin{figure}[t!]
\sidecaption{}
\centering\includegraphics[width=120pt]{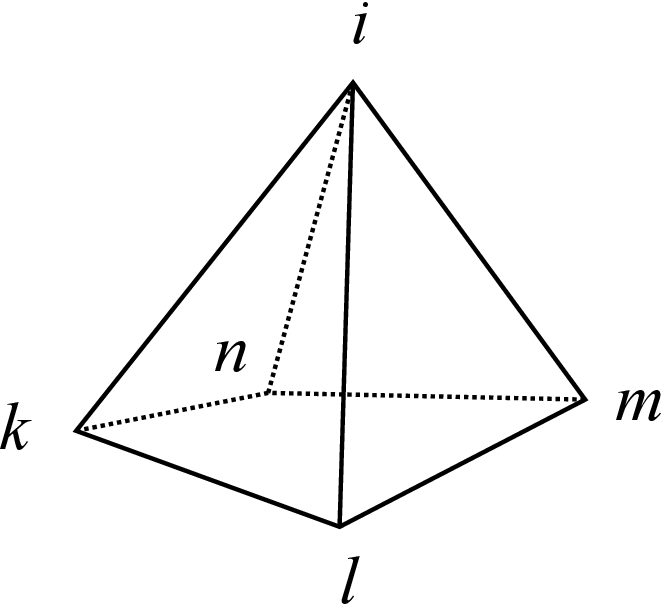}
\caption{A quadrilateral pyramid}
\label{fig:delauney_pentahedron}
\end{figure}

If there are two different spheres with five points on each of them, then there is no simplex inscribed into the both spheres (otherwise its vertices would belong to one plane which contains the intersection of the spheres). Hence, the simplices inscribed into the first spheres and the simplices inscribed into the second one have no common internal points. For each sphere we can suppose that the faces of the convex hull of the inscribed simplices has only triangular faces; in the other case, that would be a configuration of codimension greater than two. So the convex hulls can have at most one common face, so we can transform them independently. This gives us a commutation relation
\begin{equation}a_{ij,klm} a_{i'j',k'l'm'}=a_{i'j',k'l'm'}a_{ij,klm},\end{equation}
where
\begin{equation}|\{i,j,k,l,m\}\cap\{i',j',k',l',m'\}|<4,\end{equation}
\begin{equation}|\{i,j\}\cap\{i',j',k',l',m'\}|<2,$$ and $$|\{i',j'\}\cap\{i,j,k,l,m\}|<2.\end{equation}

Consider now the case of six points on one sphere. The convex hull of these points is a convex polyhedron. The polyhedron must have only triangular faces; otherwise, there is an additional linear condition (four points lie on one plane) which raises the codimension beyond $2$. There are two such polyhedra, see \cref{fig:delauney_octahedra}.

\begin{figure}{ht}
\sidecaption{}
\centering\includegraphics[width=100pt]{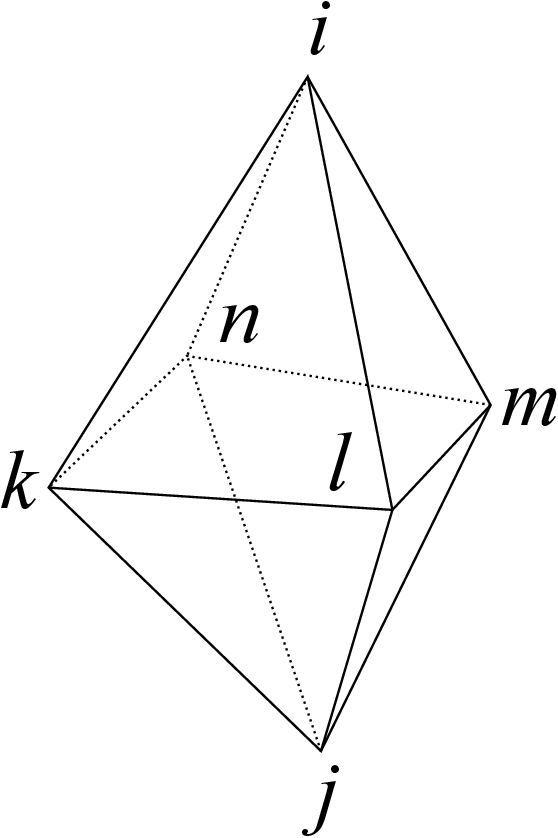}\quad\centering\includegraphics[width=70pt]{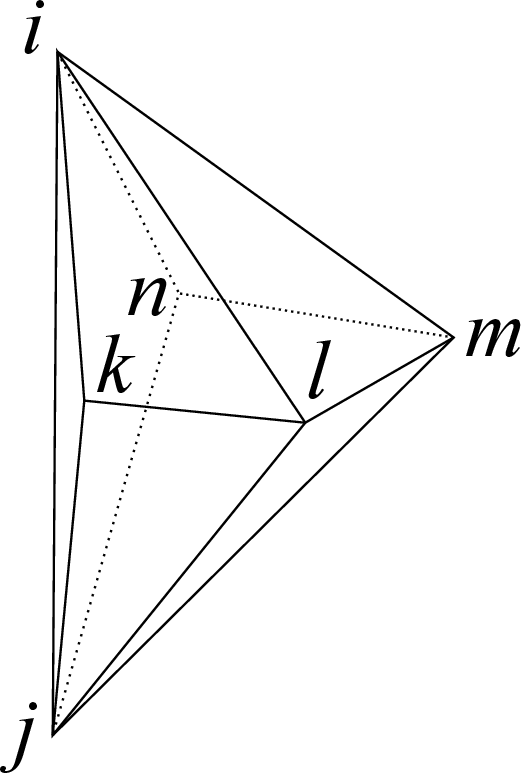}
\caption{Convex polyhedra with triangular faces and $6$ vertices}
\label{fig:delauney_octahedra}
\end{figure}

For the octahedron on the left we specify the geometrical configuration assuming that the orthogonal projection along the direction $ij$ maps the points $i$ and $j$ near the projection of the edge $kl$ and the line $ln$ is higher than $km$ if you look from $i$ to $j$ (\cref{fig:delauney_octahedron1_projection}). Note that we live in $\mathbb{R}^3$ which we may consider as oriented. In this case we have six triangulations:
\begin{enumerate}
\item $ijkl, ijkn, ijlm, ijmn$;
\item $iklm, ikmn, jklm, jkmn$;
\item $ikln, ilmn, jkln, jlmn$;
\item $ijkl, ijkm, ijlm, ikmn, jkmn$;
\item $ijkl, ijkn, ijln, ilmn, jlmn$;
\item $ikln, ilmn, jklm, jkmn, klmn$.
\end{enumerate}

\begin{figure}
\sidecaption{}
\centering\includegraphics[width=120pt]{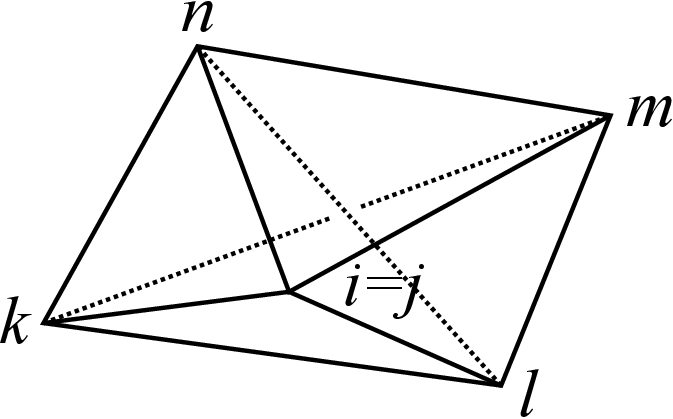}
\caption{An orthogonal projection of the octahedron}
\label{fig:delauney_octahedron1_projection}
\end{figure}

The first three have four simplices each, the last three have five simplices each. The Pachner moves between the triangulations are shown in \cref{fig:delauney_octahedron1_hasse}. Thus, we have the relation
\begin{equation}
a^{}_{km,ijn}a^{-1}_{ij,klm}a^{}_{ln,ikm}a^{-1}_{km,jln}a^{}_{ij,kln}a^{-1}_{ln,ijm}=1.
\end{equation}

\begin{figure}
\sidecaption{}
\centering\includegraphics[width=120pt]{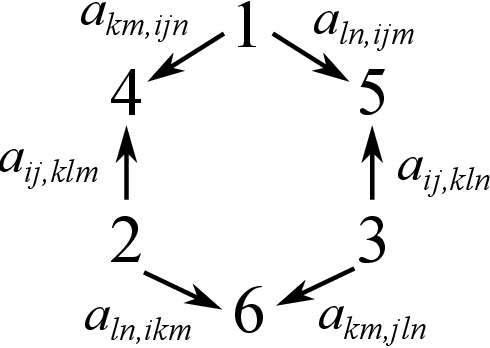}
\caption{Triangulation graph for the octahedron}
\label{fig:delauney_octahedron1_hasse}
\end{figure}

In the case of the shifted octahedron (\cref{fig:delauney_octahedra}, right) we assume the line $ln$ lies higher than the line $km$ when one looks from the vertex $i$ to the vertex $j$ (\cref{fig:delauney_octahedron2_projection}). Then we have six triangulations:
\begin{enumerate}
\item $ijkl, ijkm, ijmn$;
\item $ijkl, ijln, ilmn, jlmn$;
\item $ijkm, ijmn, iklm, jklm$;
\item $ijkn, ikln, ilmn, jkln, jlmn$;
\item $ijkn, iklm, ikmn, jklm, jkmn$;
\item $ijkn, ikln, ilmn, jklm, jkmn, klmn$.
\end{enumerate}

\begin{figure}
\sidecaption{}
\centering\includegraphics[width=100pt]{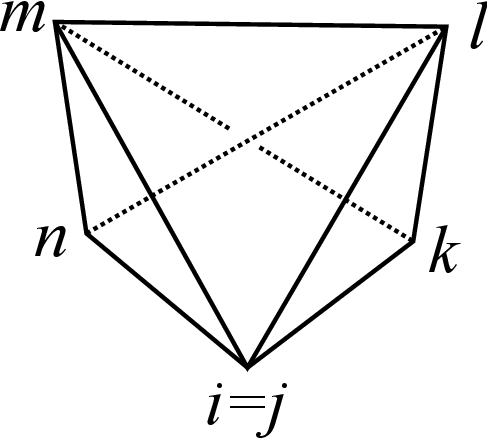}
\caption{An orthogonal projection of the shifted octahedron}
\label{fig:delauney_octahedron2_projection}
\end{figure}

The Pachner moves between the triangulations are shown in \cref{fig:delauney_octahedron2_hasse}. This gives us the relation
\begin{equation}
a_{ln,ijm}a_{kn,ijl}a_{km,jln}=a_{km,ijl}a_{kn,ijm}a_{ln,ikm}.
\end{equation}

\begin{figure}
\sidecaption{}
\centering\includegraphics[width=120pt]{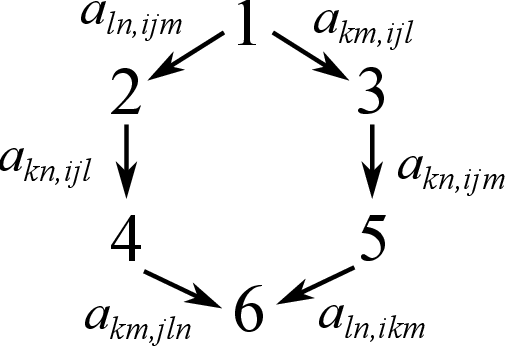}
\caption{Triangulation graph for the shifted octahedron}
\label{fig:delauney_octahedron2_hasse}
\end{figure}

%Now we can give the definition for $\Gamma_n^5$.

\begin{definition}\label{def:Gamma_n5}
The group $\Gamma_n^5$ is the group with generators $$\{ a_{ij,klm}\,|\, \{i,j,k,l,m\}\in \bar n, |\{i,j,k,l,m\}|=5\}$$ and relations
\begin{enumerate}
\item $a_{ij,klm}=a_{ji,klm}=a_{ij,lkm}=a_{ij,kml}$,
\item $a_{ij,klm} a_{i'j',k'l'm'}=a_{i'j',k'l'm'}a_{ij,klm}$, for $$|\{i,j\}\cap\{i',j',k',l',m'\}|<2,\quad |\{k,l,m\}\cap\{i',j',k',l',m'\}|<3$$ and
$$|\{i',j'\}\cap\{i,j,k,l,m\}|<2,\quad |\{k',l',m'\}\cap\{i,j,k,l,m\}|<3$$
%\item $a^{}_{km,iln} a^{-1}_{ln,jkm}=a^{}_{km,jln}a^{-1}_{ln,ikm}$ for distinct $i,j,k,l,m,n$
\item $a^{}_{km,ijn}a^{-1}_{ij,klm}a^{}_{ln,ikm}a^{-1}_{km,jln}a^{}_{ij,kln}a^{-1}_{ln,ijm}=1$ for distinct $i,j,k,l,m,n$,
\item $a_{ln,ijm}a_{kn,ijl}a_{km,jln}=a_{km,ijl}a_{kn,ijm}a_{ln,ikm}$, for distinct $i,j,k,l,m,n$.
\end{enumerate}
\end{definition}

The group $\Gamma_n^5$ can be used to construct invariants of braids and dynamic systems. %in a way similar to the groups $G_n^k$. 
Consider the configuration space $\tilde C_n(\mathbb{R}^3)$ which consists of nonplanar $n$-tuples $(x_1,x_2,\dots,x_n)$ of points in $\mathbb{R}^3$ such that for all distinct $i,j,k,l$ the points $x_i,x_j,x_k,x_l$ do not lie on the same circle.
%The fundamental group of this space
% to be the configuration space of ordered sets of $n$ distinct points in $X$. Let $\Delta$ be the interior of some fixed simplex in $\R^3$ with vertices $y_1, y_2, y_3, y_4$. Since $\Delta$ is homeomorphic to $\R^3$, the space $C_n(\Delta)$ is homeomorphic to $C_n(\R^3)$ and the groups $\pi_1(C_n(\Delta))$ and $\pi_1(C_n(\R^3))$ are isomorphic.

We construct a homomorphism from $\pi_1(\tilde C_n(\mathbb{R}^3))$ to $\Gamma_{n}^5$. Let
$$\alpha=(x_1(t),\dots, x_n(t)),\quad t\in [0,1],$$
be a loop in $\tilde C_n(\mathbb{R}^3)$. For any $t$ the set
$\mathbf{x}(t)=(x_1(t), \dots, x_n(t))$
determines a Delaunay triangulation $T(t)$ of the polytope $\conv \mathbf{x}(t)$. If $\alpha$ is in general position there will be a finite number of moments $0<t_1<t_2<\dots<t_N<1$ when the combinatorial structure of $T(t)$ changes, and for each $p$ the transformation of the triangulation at the moment $t_p$ will be the Pachner move on simplices with vertices $i_p,j_p,k_p,l_p,m_p$. We assign to this move the generator $a_{i_pj_p,k_pl_pm_p}$ or $a^{-1}_{i_pj_p,k_pl_pm_p}$ and denote
$\phi(\alpha)=\prod_{p=1}^N a_{i_pj_p,k_pl_pm_p}\in\Gamma_{n}^5.$

\begin{theorem}\label{thm:Gamman5_invariant}
The homomorphism $\phi\colon \pi_1(\tilde C_n(\mathbb{R}^3))\to\Gamma^5_{n}$ is well defined.
\end{theorem}

\begin{proof}
We need to show that the element $\phi(\alpha)$ does not depend on the choice of the representative $\alpha$ in a given homotopy class. Given a homotopy in general position $\alpha(\tau), \tau\in[0,1],$  of loops in $\tilde C_n(\mathbb{R}^3)$, the transformations of the words $\phi(\alpha(\tau))$ correspond to point configurations of codimension $2$, considered above the definition of $\Gamma_n^5$, and thus are counted by the relations in the group $\Gamma^5_{n}$. Therefore, the element $\phi(\alpha(\tau))$ of the group $\Gamma^5_{n}$ remains the same when $\tau$ changes.
\end{proof}

\subsection{General case}
\label{sec:gamma_general}

The braids in higher dimensional Euclidean spaces were defined by the second named author in~\cite{HigherGnk}.
For our needs we modify that definition slightly.

Consider the configuration space $\tilde C_n(\mathbb{R}^{k-2})$, $4\le k\le n$, consisting of $n$-point configurations $\mathbf{x}=(x_1, x_2,\dots, x_n)$ in $\mathbb{R}^{k-2}$, such that $\dim\conv\mathbf{x}=k-2$ and there are no $k-1$ points which lie on one $(k-4)$-dimensional sphere (the intersection of a $(k-3)$-sphere and a hyperplane in in $\mathbb{R}^{k-2}$).

A configuration $\mathbf{x}=(x_1, x_2,\dots, x_n)\in\tilde C_n(\mathbb{R}^{k-2})$ determines a Delaunay triangulation of $\conv\mathbf{x}$ which is unique when $\mathbf x$ is in general position. If the vertices $x_{i_1},\dots,x_{i_{k-1}}$ span a simplex of the Delaunay triangulation then the interior of the circumscribed sphere over these points does not contain other points of the configuration. The inverse statement is true when $\mathbf x$ is generic. The condition that no $k-1$ points lie on one $(k-4)$-dimensional sphere ensures that there are no degenerate simplices in the Delaunay triangulation.

Let $\alpha=(x_1(t),\dots, x_n(t)),\quad t\in [0,1],$
be a loop in $\tilde C_n(\mathbb{R}^{k-2})$. For any $t$ the configuration $\mathbf{x}(t)=(x_1(t), \dots, x_n(t))$
determines a Delaunay triangulation $T(t)$ of the polytope $\conv \mathbf{x}(t)$. If $\alpha$ is generic then there will be a finite number of moments $0<t_1<t_2<\dots<t_L<1$ when the combinatorial structure of $T(t)$ changes. We shall call such configurations {\em singular}.

%We shall assume that the combinatorial structure of the boundary $\partial\conv(\bx(t))$ does not change with $t$, i.e. it coincides with the initial structure of $\partial\conv(\bx(0))$. The set of configurations that obey this property on the combinatorial structure of the convex boundary we denote by $\tilde C_n(\R^{k-2})$. Thus, we suppose $\bx(t)\in\tilde C_n^{\bx(0)}(\R^{k-2})$

For each singular configuration $\mathbf{x}(t_i)$, either one simplex degenerates and disappears on the boundary $\partial\conv \mathbf{x}(t)$ or the Delaunay triangulation is not unique. That means that there is a sphere in $\mathbb{R}^{k-2}$ which contains $k$ points of $\mathbf x$ on it and no points of $\mathbf x$ inside. Assuming $\mathbf x$ is generic, the span of these $k$ points is a simplicial polytope. Below we shall count only the latter type of singular configurations.

The simplicial polytopes in $\mathbb{R}^{k-2}$ with $k$ vertices are described in~\cite{Gruenbaum}.
Each of them is the join $\Delta_P*\Delta_Q$ of simplices $\Delta_P=\conv(x_{p_1},\dots,x_{p_l})$ of dimension $l-1\ge 1$ and $\Delta_Q=\conv(x_{q_1},\dots,x_{q_{k-l}})$ of dimension $k-l-1\ge 1$ such that the intersection $\relint(x_{p_1},\dots,x_{p_l})\cap\relint(x_{q_1},\dots,x_{q_{k-l}})$ consists of one point.

We recall that the {\em join} of two sets $X,Y\subset\mathbb{R}^m$ is defined as
$X*Y=\{ \lambda x+(1-\lambda)y\,|\, x\in X, y\in Y, \lambda\in [0,1]\},$
and the {\em relative interior} of a finite set $X\subset\mathbb{R}^m$ is defined as
$$\relint X=\left\{ \sum_{x\in X}\lambda_x x\,|\,\forall x\ \lambda_x>0, \sum_{x\in X}\lambda_x=1\right\}.$$
The polytope $\Delta_P*\Delta_Q$ has two triangulations:
$$T_P=\{\mathbf{x}_{P\cup Q}\setminus\{x_p\}\}_{p\in P} = \left\{x_{p_1}\cdots x_{p_{i-1}}x_{p_{i+1}}\cdots x_{p_l}x_{q_1}\cdots x_{q_{k-l}}\right\}_{i=1}^l$$
and
$$T_Q=\{\mathbf{x}_{P\cup Q}\setminus\{x_q\}\}_{q\in Q} = \left\{x_{p_1}\cdots x_{p_l} x_{q_1} \cdots x_{q_{i-1}} x_{q_{i+1}} \cdots x_{q_{k-l}}\right\}_{i=1}^{k-l}.$$
Here $P=\{p_1,\dots,p_l\}$, $Q=\{q_1,\dots,q_{k-l}\}$ and $\mathbf{x}_J=\{x_j\}_{j\in J}$ for any $J\subset\{1,\dots,n\}$.

The condition $\relint(\mathbf{x}_P)\cap\relint(\mathbf{x}_Q)=\{z\}$ implies $P\cap Q=\emptyset$. Thus, when the configuration $\bx(t)$ goes over a singular value $t_i, i=1,\dots, L$, in the Delaunay triangulation simplices $T_{P_i}$ are replaced by simplices $T_{Q_i}$ for some subsets $P_i,Q_i\subset\{1,\dots,n\}$, $P_i\cap Q_i=\emptyset$, $|P_i|,|Q_i|\ge 2$, $|P_i\cup Q_i|=k$. This transformation is called a {\em Pachner move}. The pair of subsets $(P_i,Q_i)$ we call the {\em type}\label{def:type_Pacher_move} of the Pachner move.  We assign to the transformation the letter $a_{P_i,Q_i}$.

Hence, the loop $\alpha$ produces a word
\begin{equation}\label{eq:Gamma_nk_braid_invariant}
\Phi(\alpha)=\prod_{i=1}^L a_{P_i,Q_i}
\end{equation}
 in the alphabet
%\begin{equation*}%\label{eq:gamma_nk_generators}
$\mathcal A_n^k=\left\{ a_{P,Q}\,|\, P,Q\subset\{1,\dots,n\}, P\cap Q=\emptyset, |P\cup Q|=k, |P|,|Q|\ge 2\right\}.$
%\end{equation*}

Now let us consider a generic homotopy $\alpha_s, s\in[0,1]$ between two generic loops $\alpha_0$ and $\alpha_1$. A loop $\alpha_s=\{\bx(s,t)\}_{t\in[0,1]}$ can contain a configuration of codimension $2$. This means that for some $t$ the configuration $\bx(s,t)$ has two different $k$-tuples of points, such that each of them lies on a sphere whose interior contains no points of $\bx(s,t)$. If these spheres do not coincide then their intersection contains at most $k-2$ points (the intersection can not contain $k-1$ points because $\bx(s,t)\in\tilde C_n(\mathbb{R}^{k-2})$). Hence, the simplices involved in one Pachner move can not be involved in the other one, so the Pachner moves can be fulfilled in any order.

If  the $k$-tuples of points lie on one sphere, there is a sphere with $k+1$ points of $\bx(s,t)$ on it and its interior contains no points of $\bx(s,t)$. These $k+1$ points span a simplicial polytope in $\mathbb{R}^{k-2}$. Such polytopes are described in~\cite{Gruenbaum}. The description uses the notion of {\em Gale transform}.

Let $X=\{x_1,\dots,x_n\}$ be a set of $n$ points in $\mathbb{R}^d$ such that $\dim\conv X = d$. Then $n\ge d+1$. Let $x_i=(x_{1i},\dots,x_{di})\in\mathbb{R}^d$, $i=1,\dots,n$, be the coordinates of the points of $X$. The matrix
$$
M = \left(\begin{array}{cccc}
x_{11} & x_{12} & \cdots & x_{1n}\\
x_{21} & x_{22} & \cdots & x_{2n}\\
\cdots &\cdots &\cdots &\cdots \\
x_{d1} & x_{d2} & \cdots & x_{dn}\\
1 & 1 & \cdots & 1
\end{array}\right)
$$
has rank $d+1$. Then the dimension of the space $\ker M = \{ b\in\mathbb{R}^n\,|\, Mb=0\}$ of dependencies between columns of $M$ is equal to $n-(d+1)$. Take any basis $b_j = (b_{j1},b_{j2},\dots b_{jn})$, $j=1,\dots, n-d-1,$ of $\ker M$ and write it in matrix form
\begin{equation}\label{eq:gale_b_matrix}
B = \left(\begin{array}{ccc}
b_{11} & \cdots & b_{1n}\\
\cdots &\cdots  &\cdots \\
b_{n-d-1,1} & \cdots & b_{n-d-1,n}
\end{array}\right).
\end{equation}
The columns of the matrix $B$ form a set $Y={y_1,\dots, y_n}$, $y_i=(b_{1i},\dots,b_{n-d-1,i}),$ in $\mathbb{R}^{n-d-1}$. The set $Y$ is called a {\em Gale transform} of the point set $X$. Gale transforms which correspond to different bases of $\ker M$ are linearly equivalent. The vectors of the Gale transform $Y$ may coincide.

%\begin{example}
%Let $X$ be a pentagon in $\mathbb{R}^2$ with vertices $x_1=(0,2)$, $x_2=(-2,1)$, $x_3=(-1,-1)$, $x_4=(1,-1)$, $x_5=(2,1)$, see \cref{fig:pentagon_gale_example} left. Then
%$$
%M = \left(\begin{array}{ccccc}
%0 & -2 & -1 & 1 & 2\\
%2 & 1 & -1 & -1 & 1\\
%1 & 1 & 1 & 1 & 1
%\end{array}\right)
%$$
%and
%$$
%B = \left(\begin{array}{ccccc}
%-4 & 1 & 3 & -5 & 5\\
%-4 & 6 & -7 & 5 & 0
%\end{array}\right).
%$$
%The Gale transform $Y$ consists of vectors $y_1=(-4,-4)$, $y_2=(1,6)$, $y_3=(3,-7)$, $y_1=(-5,5)$, $y_1=(5,0)$, see \cref{fig:pentagon_gale_example} right.
%
%\begin{figure}
%\centering\includegraphics[width=120pt]{pentagon_configuration.eps} \hspace{10mm} \includegraphics[width=120pt]{pentagon_gale_transformation.eps}
%\caption{A pentagon and its Gale transform}\label{fig:pentagon_gale_example}
%\end{figure}
%\end{example}

Let $Y={y_1,\dots, y_n}$ be a Gale transform of $X$. The set
$$\bar Y = \{\bar y_1,\dots, \bar y_n\},\quad \bar y_i=\left\{\begin{array}{cl}\frac{y_i}{\|y_i\|}, & y_i\ne 0,\\ 0, & y_i=0 \end{array}\right.
$$
is called a {\em Gale diagram} of the point set $X$. It is a subset of $S^{n-d-2}\cup\{0\}$.

Denote $N=\{1,\dots,n\}$. Two subsets $\bar Y = \{\bar y_1,\dots, \bar y_n\}$ and $\bar Y' = \{\bar y'_1,\dots, \bar y'_n\}$ in $S^{n-d-2}\cup\{0\}$ are called {\em equivalent} if there is a permutation $\sigma$ of $N$ such that for any $J\subset N$
$$
0\in\relint \bar Y_J \Longleftrightarrow 0\in\relint \bar Y'_{\sigma(J)}.
$$
Here we denote $\bar Y_J = \{\bar y_i\}_{i\in J}$ and $\bar Y'_J = \{\bar y'_i\}_{i\in J}$.

The properties of Gale diagrams can be summarized as follows~\cite{Gruenbaum}.
\begin{theorem}\label{thm:gale_diagram}
\begin{enumerate}
\item Let $X$ be a set of $n$ points which are vertices of some polytope $P$ in $\mathbb{R}^d$ and $\bar Y$ be its Gale diagram. Then the set of indices $J\subset N$ defines a face of $P$ if and only if $0\in\relint Y_{N\setminus J}$.
\item Let $X$ and $X'$ be sets of vertices of polytopes $P$ and $P'$, $|X|=|X'|$, and $\bar Y$ and $\bar Y'$ be their Gale diagrams. Then $P$ and $P'$ are combinatorially equivalent if and only if $\bar Y$ and $\bar Y'$ are equivalent.
\item For any $n$-point set $\bar Y\in S^{n-d-2}\cup\{0\}$ such that $\bar Y$ spans $\mathbb{R}^{n-d-1}$ and $0$ lies in the interior of $\conv\bar Y$, there is an $n$-point set $X$ in $\mathbb{R}^d$ such that $\bar Y$ is a Gale diagram of $X$.
\end{enumerate}
\end{theorem}

The theorem implies (see~\cite{Gruenbaum})
that simplicial polytopes with $k+1$ vertices in $\mathbb{R}^{k-2}$ are in a bijection with standard Gale diagrams in $\mathbb{R}^2$ (defined uniquely up to isometries of the plane).

A {\em standard Gale diagram} of order $l=k+1$ is a subset $\bar Y$, $|\bar Y|=l$, of the vertices set $\{e^{\frac{\pi i p}{l}}\}_{p=0}^{2l-1}$ of the regular $2l$-gon inscribed in the unit circle $S^1$, such that:
\begin{enumerate}
\item any diameter of $S^1$ contains at most one point of $\bar Y$;
\item for any diameter of $S^1$, any open half-plane determined by it contains at least two points of $\bar Y$.
\end{enumerate}

The first property means the corresponding polytope is simplicial, the second means any of the $k+1$ vertices of the polytope is a face.

The number $c_l$ of standard Gale diagrams of order $l$ is equal to
$$
c_l= 2^{\left[\frac{l-3}{2}\right]}-\left[\frac{l+1}4\right]+\frac 1{4l}\sum_{h\colon 2\nmid h\mid l}\varphi(h)\cdot 2^{\frac lh},
$$
where $l=2^{a_0}\prod_{i=1}^{t}p_i^{a_i}$ is the prime decomposition of $l$, and $\varphi$ is Euler's function. For small $l$ the numbers are $c_5=1, c_6=2, c_7=5$, see \cref{fig:gale_standard_diagrams}.

\begin{figure}
\centering
\begin{tabular}{cc}
\includegraphics[height=50pt]{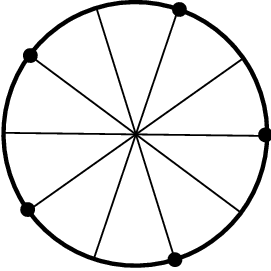} & \includegraphics[height=50pt]{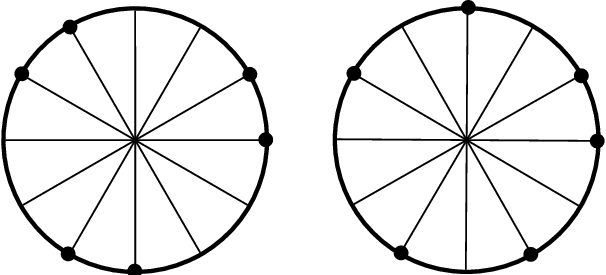}\\
 $l=5$ & $l=6$\\
 \multicolumn{2}{c}{\includegraphics[height=60pt]{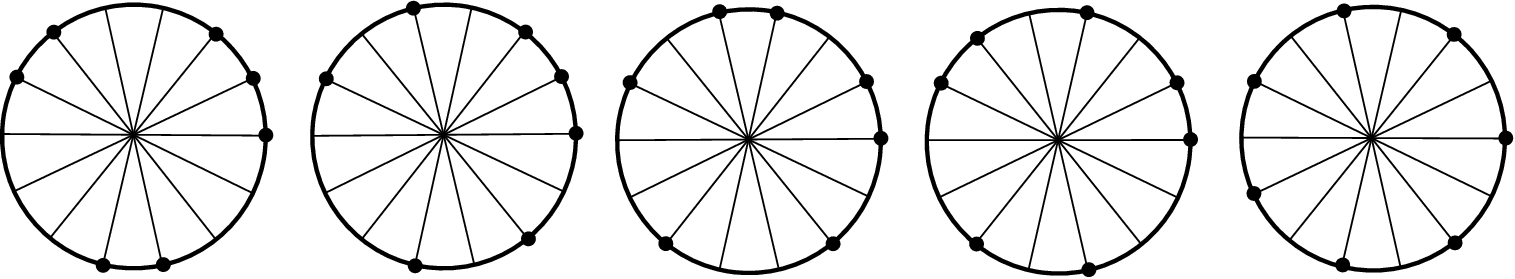}}\\
 \multicolumn{2}{c}{$l=7$}
\end{tabular}
\caption{Standard Gale diagrams of small order}\label{fig:gale_standard_diagrams}
\end{figure}

Let us describe the triangulations of the simplicial polytopes with $k+1$ vertices in $\mathbb{R}^{k-2}$.

Let $X=\{x_1,\dots,x_n\}$ be a subset in $\mathbb{R}^d$, so $x_i = (x_{i1},\dots, x_{id})$, $i=1,\dots, n$. Let $P=\conv X$ be the convex hull of $X$, assume that $\dim P = d$. A triangulation $T$ of the polytope $P$ with vertices in $X$ is called {\em regular}\label{def:regular_triangulation} if there is a height function $h\colon X\to\mathbb{R}$ such that $T$ is the projection of the lower convex of the lifting $X^h=\{x_1^h,\dots,x_n^h\}\subset\mathbb{R}^{d+1}$, where $x_i^h = (x_i, h(x_i))$. This means that a set of indices $J\subset\{1,\dots,n\}$ determines a face of $T$ if and only if there exists a linear functional $\phi$ on $\mathbb{R}^{d+1}$ such that $\phi(0,\dots,0,1)>0$ and $J=\{i\,|\, \phi(x^h_i) = \min_{x^h\in X^h}\phi(x^h)\}$. In case $T$ is regular we write $T=T(X,h)$. Any generic height function induces a regular triangulation.

The Delaunay triangulation of $X$ is regular with the height function $h\colon \mathbb{R}^d\to\mathbb{R}$, $h(z)=\|z\|^2= \sum_{i=1}^d z_i^2$ if $z=(z_1,\dots,z_d)\in\mathbb{R}^d$.

A height function $h\colon X\to\mathbb{R}$ can be regarded as a vector $h=(h_1,\dots, h_n)\in\mathbb{R}^n$ where $h_i=h(x_i)$. Denote $\beta(h)=Bh\in\mathbb{R}^{n-d-1}$ where $B$ is the matrix~\eqref{eq:gale_b_matrix} used to define a Gale transform of $X$. Let $\bar Y=\{\bar y_1,\dots,\bar y_n\}$ be a Gale diagram of $X$. Convex cones generated by the subsets of $\bar Y$ split the space $\mathbb{R}^{n-d-1}$ into a union of conic cells. A relation between the triangulation $T(X,h)$ and the Gale diagram $\bar Y$ can be described as follows.%~\cite{LRS}.

\begin{theorem}\label{thm:gale_triangulation}
\begin{enumerate}
\item If $T(X,h)$ is a regular triangulation of $X$ then $\beta(h)$ belongs to a conic cell of maximal dimension in the splitting of $\mathbb{R}^{n-d-1}$ induced by $\bar Y$.
\item Let $J\subset N=\{1,\dots,n\}$. The set $X_J$ spans a cells of the triangulation $T(X,h)$ if and only if $\beta(h)\in\mathrm{concone}(\bar Y_{N\setminus J})$, where $\mathrm{concone}(X_{N\setminus J})$ is the convex cone spanned by the set $\bar Y_{N\setminus J}$.
\end{enumerate}
\end{theorem}

Let $P$ be a simplicial polytopes with $l=k+1$ vertices in $\mathbb{R}^{l-3}$ and $\bar Y=\{\bar y_1,\dots,\bar y_l\}$ be the corresponding standard Gale diagram. By \cref{thm:gale_triangulation} there are $l$ different regular triangulation which correspond to open sectors between the rays spanned by the vectors of $\bar Y$. The graph whose vertices are combinatorial classes of triangulations of $P$ and the edges are Pachner moves, is a cycle. Let us find which Pachner moves appear in this cycle.

We change the order of vertices of $P$ (and hence, the order of the points of $\bar Y$) so that the points $\bar y_1,\dots,\bar y_l$ appear in this sequence when one goes counterclockwise on the unit circle. For each $i$ denote $R_{\bar Y}(i)$ (respectively, $L_{\bar Y}(i)$) be the set of indices $j$ of vectors $\bar y_j$ that lie in the right (respectively, left) open half-plane incident to the oriented line spanned by the vector $\bar y_i$. Then the Pachner move which occurs when the vector $\beta(h)$ passes $\bar y_i$ from right to left, will be marked with the letter $a_{R_{\bar Y}(i),L_{\bar Y}(i)}\in \mathcal{A}_l^{l-1}$. The Pachner moves of the whole cycle of triangulation give the word $w_{\bar Y}=\prod_{i=1}^l a_{R(i),L(i)}$.

Thus, we have the relation $w_{\bar Y}=1$ that we should impose in order to make the words like $\Phi(\alpha)$ independent on resolutions of configurations of codimension $2$.

\begin{example}
Consider the standard Gale diagram of order $5$, \cref{fig:gale_d2_marked}. Then we have $R(1)=\{4,5\}$, $L(1)=\{2,3\}$, $R(2)=\{1,5\}$, $L(2)=\{3,4\}$ etc. The corresponding word is equal to
$$
w =a_{45,23}a_{15,34}a_{12,45}a_{23,15}a_{34,12}.
$$
\begin{figure}
\sidecaption{}
\centering\includegraphics[width=80pt]{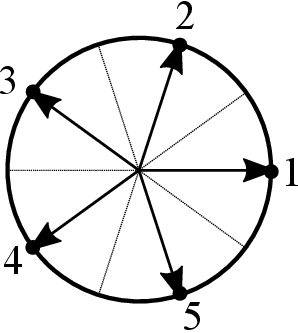}
\caption{Standard Gale diagram of order $5$}\label{fig:gale_d2_marked}
\end{figure}
\end{example}

We can give now the definition of $\Gamma_n^k$ groups.

\begin{definition}\label{def:Gamma_nk}
Let $4\le k\le n$. The group $\Gamma_n^k$ is the group with %$n\choose k\cdot(2^k-2k-2)$
the generators
$$
a_{P,Q},\quad  P,Q\subset\{1,\dots,n\}, P\cap Q=\emptyset, |P\cup Q|=k, |P|,|Q|\ge 2,
$$
and the relations:
\begin{enumerate}
\item $a_{Q,P}=a_{P,Q}^{-1}$;
\item {\em far commutativity}: $a_{P,Q}a_{P',Q'}=a_{P',Q'}a_{P,Q}$ for each generators $a_{P,Q}$, $a_{P',Q'}$ such that
\begin{align*}
|P\cap(P'\cup Q')|&<|P|, & |Q\cap(P'\cup Q')|&<|Q|,\\
 |P'\cap(P\cup Q)|&<|P'|, & |Q'\cap(P\cup Q)|&<|Q'|;
\end{align*}
\item {\em $(k+1)$-gon relations}: for any standard Gale diagram $\bar Y$ of order $k+1$ and any subset $M=\{m_1,\dots,m_{k+1}\}\subset\{1,\dots,n\}$
$$ \prod_{i=1}^{k+1} a_{M_R(\bar Y,i), M_L(\bar Y,i)}=1,$$
where $M_R(\bar Y,i) = \{m_j\}_{j\in R_{\bar Y}(i)}$,  $M_L(\bar Y,i) = \{m_j\}_{j\in L_{\bar Y}(i)}$.
\end{enumerate}
\end{definition}

\begin{example}
Let us write the $(k+1)$-gon relations in $\Gamma_n^k$ for small $k$.

The group $\Gamma_n^4$ has one pentagon relation
$$
a_{m_4m_5,m_2m_3}a_{m_1m_5,m_3m_4}a_{m_1m_2,m_4m_5}a_{m_2m_3,m_1m_5}a_{m_3m_4,m_1m_2}=1.
$$
Our definition of $\Gamma_n^4$ differs slightly from the definition of these groups in~\cite{ManturovKim}. To obtain the groups defined in~\cite{ManturovKim} we should add relations
$a_{m_1m_2,m_3m_4}^2=1$ and
\begin{gather*}
a_{m_1m_2,m_3m_4}a_{m_1m_2,m_5m_6}=a_{m_1m_2,m_5m_6}a_{m_1m_2,m_3m_4},\\
a_{m_1m_2,m_3m_4}a_{m_1m_5,m_2m_6}=a_{m_1m_5,m_2m_6}a_{m_1m_2,m_3m_4},
\end{gather*}
where $|\{m_3,m_4\}\cap\{m_5,m_6\}|\le 1$.

The group $\Gamma_n^5$ has two hexagon relations
\begin{gather*}
a_{m_5m_6,m_2m_3m_4}a_{m_1m_5m_6,m_3m_4}a_{m_1m_2,m_4m_5m_6}a_{m_1m_2m_3,m_5m_6}a_{m_3m_4,m_1m_2m_6}a_{m_3m_4m_5,m_1m_2}=1,\\
a_{m_5m_6,m_2m_3m_4}a_{m_1m_5m_6,m_3m_4}a_{m_1m_2m_6,m_4m_5}a_{m_1m_2m_3,m_5m_6}a_{m_3m_4,m_1m_2m_6}a_{m_4m_5,m_1m_2m_3}=1
\end{gather*}
as we have seen in the previous subsection.

%The group $\Gamma_n^6$ has five heptagon relations, etcetera.
%\begin{multline*}
%a_{m_6m_7,m_2m_3m_4m_5}a_{m_1m_6m_7,m_3m_4m_5}a_{m_1m_2m_6m_7,m_4m_5}a_{m_1m_2m_3,m_5m_6m_7}\cdot\\
%\shoveright{a_{m_1m_2m_3m_4,m_6m_7}a_{m_4m_5,m_1m_2m_3m_7}a_{m_4m_5m_6,m_1m_2m_3}=1,}\\
%\shoveleft{a_{m_6m_7,m_2m_3m_4m_5}a_{m_1m_6m_7,m_3m_4m_5}a_{m_1m_2m_6m_7,m_4m_5}a_{m_1m_2m_3m_7,m_5m_6}\cdot}\\
%\shoveright{a_{m_1m_2m_3m_4,m_6m_7}a_{m_4m_5,m_1m_2m_3m_7}a_{m_5m_6,m_1m_2m_3m_4}=1,}\\
%\shoveleft{a_{m_6m_7,m_2m_3m_4m_5}a_{m_1m_6m_7,m_3m_4m_5}a_{m_1m_2m_7,m_4m_5m_6}a_{m_1m_2m_3m_7,m_5m_6}\cdot}\\
%\shoveright{a_{m_1m_2m_3m_4,m_6m_7}a_{m_3m_4m_5,m_1m_2m_7}a_{m_5m_6,m_1m_2m_3m_4}=1,}\\
%\shoveleft{a_{m_6m_7,m_2m_3m_4m_5}a_{m_1m_6m_7,m_3m_4m_5}a_{m_1m_2m_7,m_4m_5m_6}a_{m_1m_2m_3,m_5m_6m_7}\cdot}\\
%\shoveright{a_{m_1m_2m_3m_4,m_6m_7}a_{m_3m_4m_5,m_1m_2m_7}a_{m_4m_5m_6,m_1m_2m_3}=1,}\\
%\shoveleft{a_{m_5m_6m_7,m_2m_3m_4}a_{m_1m_6m_7,m_3m_4m_5}a_{m_1m_2m_7,m_4m_5m_6}a_{m_1m_2m_3,m_5m_6m_7}\cdot}\\
%a_{m_2m_3m_4,m_1m_6m_7}a_{m_1m_4m_5,m_1m_2m_3m_7}a_{m_4m_5m_6,m_1m_2m_3}=1.
%\end{multline*}
\end{example}

%For groups $\Gamma_n^k$ we have a generalisation of Theorem~\cref{thm:Gamman5_invariant}.
\begin{theorem}\label{thm:Gamma_nk_invariant}
Formula~\eqref{eq:Gamma_nk_braid_invariant} defines a correct homomorphism $$\Phi\colon\pi_1(\tilde C_n(\mathbb{R}^{k-2}))\to\Gamma_n^k.$$
\end{theorem}
\begin{proof}
We need to show that the element $\Phi(\alpha)$ does not depend on the choice of the representative $\alpha$ in a given homotopy class. Given a generic homotopy $\alpha(\tau), \tau\in[0,1],$  of loops in $\tilde C_n(\mathbb{R}^{k-2})$, the transformations of the words $\Phi(\alpha(\tau))$ correspond to point configurations of codimension $2$, considered above the definition of $\Gamma_n^k$, and thus are counted by the relations in the group $\Gamma^k_{n}$. Therefore, the element $\Phi(\alpha(\tau))$ of the group $\Gamma^k_{n}$ remains the same when $\tau$ changes.
\end{proof}

\section{Main results}

\subsection{$\Gamma^k_n$-valued invariant of braids on manifolds}

\begin{theorem} \label{thm:gamma_geometrical_invariant}
Let $(M^{d},g)$ be a manifold with a Riemannian metric and $B(M_{g},n)_{j}$, $j=1,\dots,p$ be its geometrical braid groups. Then for any $j$ there is a well defined mapping
$B(M_g,n)_j\to \Gamma_{n}^{d+2}$.
\end{theorem}

\begin{proof}
By \cref{def:geometric_n_braid}, the space of geometrical braids on the manifold $M^d$ is the fundamental group of the manifold of triangulations: $B(M_g,n)_j=\pi_1((M^{dn}_g)_j)$. Hence, we need to construct a mapping from homotopy classes of loops in $(M^{dn}_g)_j$ into the group $\Gamma_{n}^{d+2}$.

Let us fix a number $j$ and consider a loop
$\alpha = (x_1(t), \dots, x_n(t)), t\in [0,1]$
in $(M^{dn}_g)_j$, where for each $i=1,\dots, n$ the point $x_i(t)\in M^d$. For each $t$ the collection of points $(x_1(t), \dots, x_n(t))$ defines a Delaunay triangulation $T(t)$ of the manifold $M$. If the loop $\alpha$ is in general position, then there is a finite number of moments $0<t_1<\dots < t_N<1$ when the combinatorial structure of the triangulation $T(t)$ changes. Each of those moments $t_p$ corresponds to a codimension 1 configuration of points of the manifold $M^d$, and at each of them the triangulation undergoes some Pachner move, which transforms the triagulation simplices $T_{P_p}$ into simplices $T_{Q_p}$, where $P_p, Q_p \subset \{1, \dots, n\}, \, P_p\cap Q_p=\emptyset, \, |P_p|, |Q_p|\ge 2$ and $|P_p\cup Q_p|=d+2$. To each of those moments we attribute the element $a_{P_p,Q_p}$ of the group $\Gamma_n^{d+2}$.

Thus a loop $\alpha$ produces a word $\Phi(\alpha)=\prod_{p=1}^N a_{P_p,Q_p}.$

Now we need to prove that this mapping $\Phi\colon \pi((M^{dn}_g)_j) \to \Gamma_n^{d+2}$ is well defined: the element $\Phi(\alpha)$ does not depend on the choice of a representative in a given homotopy class.

Let $\alpha(\tau)$ be a generic homotopy, where $\tau\in [0,1]$ and $\alpha(0)=\alpha$. With it we associate a family of words $\Phi(\alpha(\tau))\in\Gamma_n^{d+2}$. Note that the word $\Phi(\alpha(\tau))$ changes whenever the loop passes through a point of codimension 2 in the manifold of triangulations $M^{dn}_g$. But such configurations of points in $M^d$ exactly correspond to the relations of the group $\Gamma_n^{d+2}$. Therefore the element $\Phi(\alpha(\tau))$ of the group $\Gamma_n^{d+2}$ remains the same when $\tau$ changes.
\end{proof}

Quite analogously, we get the following

\begin{theorem}
Let $M^{d}$ be a smooth manifold of dimension $d$ and $B(M,n)_j$, $j=1,\dots,q$, be its topological braid groups.
Then for any $j$ there is a well defined mapping $B(M,n)_j\to \Gamma_{n}^{d+2}$.
\end{theorem}

\begin{proof}
	By \cref{def:topological_n_braid}, the topological braids group on the manifold $M^d$ is the fundamental group of the manifold of triangulations: $B(M,n)_j=\pi_1((M^{dn}_{top})_j)$.
	
	Consider a loop $\gamma$ in the manifold $M^{dn}_{top}$ and let $\gamma$ be in general position. There is a finite number of intersections $\mathbf{x}_1,\dots, \mathbf{x}_k$ between the loop $\gamma$ and the codimension 1 stratum of the manifold of triangulations, where the triangulation undergoes a flip. Consider an intersection point $\mathbf{x}_i$ and its neighbourhood $U$. The intersection point $\mathbf{x}_i$ is a set of points $x_{1,i},\dots, x_{n,i}\in M$.
	
	Consider two metrics $g_{\alpha}$ and $g_{\beta}$ on the manifold $M$, yielding a structure of geometrical manifold of triangulations on the neighbourhood $U$. Let us denote the neighbourhood $U$ equipped with the corresponding metric by $U_{g_{\alpha}}$ and $U_{g_{\beta}}$, respectively. As it was described in \cref{sec:manifold_of_triangulations}, equivalence of strata of the manifolds of triangulation $M^{dn}_{g_{\alpha}}$ and $M^{dn}_{g_{\beta}}$ is defined by an existence of a homeomorphism $h\colon M^{dn}_{g_{\alpha}}\to M^{dn}_{g_{\beta}}$ with certain properties.
	
	The arc $l=\gamma\cap U\subset M^{dn}_{top}$ gives rise to an arc $l_1\subset U_{g_{\alpha}}$ and an arc $l_2 \subset U_{g_{\beta}}$, and $l_2=h(l_1)$ where $h$ is the homeomorphism from the definition of cell equivalence. Likewise, the intersection point $X_i\in U$ gives a point $X_i'\in U_{g_{\alpha}}$ and a point $X_i''\in U_{g_{\beta}}$. Naturally, those points are the intersections between the codimension 1 stratum of the corresponding manifold and the arcs $l_1, l_2$, respectively.
	
	At the point $\mathbf{x}_i$ a certain Pachner move occurs. The type of this move is independent of the choice of metric on the manifold. Therefore, the Pachner moves occurring in the points $\mathbf{x}_i', \mathbf{x}_i''$ of the manifolds $U_{g_{\alpha}}, U_{g_{\beta}}$ are one and the same. Choosing the numbering of the $n$ points which form the vertices of the triangulations, we hence obtain the same generator of the group $\Gamma_n^{d+2}$ corresponding to this intersection point. Therefore the letter which appears in the word $\Phi(\gamma)$ does not depend on the choice of the metric.
	
	Therefore the mapping $\Phi$ is defined on the fundamental group of the topological manifold of triangulations $M^{dn}_{top}$ and due to \cref{thm:gamma_geometrical_invariant} it does not depend on the choice of a representative in the homotopy class of a loop. Hence, it is a well defined mapping from $B(M,n)_j$ to the group $\Gamma_{n}^{d+2}$.
\end{proof}

Hence, we get invariants of smooth metrical manifolds and of topological manifolds: these invariants are the corresponding images of the fundamental group of the manifold of triangulations in the corresponding group $\Gamma^k_n$.

%One can proceed playing the same game with manifolds having other structures (complex, contact, symplectic, K\"ahler, or related to an action of some group). The definition of the corresponding manifolds of triangulations is left to the reader.

%\subsection{Braids in Euclidean spaces and several copies of $\Gamma_n^k$}
%
%It turns out that the above approach allows one to construct well defined homomorphisms
%from the braid groups $\pi_{1}(\tilde C^{n}(\mathbb{R}^{k-2}))$ not just to $\Gamma_n^{k}$ but to
%$\Gamma_n^{k} \times \cdots \times \Gamma_n^{k}$
%($n-k$ copies where $n$ is the total number of points).
%
%Indeed, consider a dynamics of points in $\mathbb{R}^{k-2}$. We shall take into account all those moments when $k$ points belong to the same sphere.
%Whenever $k$ points lie on the same sphere, one can index everything by pairs of numbers: how many points are on the one side, how many ones are on the other.
%Then we get the product $\Gamma\times\Gamma\times\Gamma\times\dots$,
%where the first gamma means ``0 points on the one hand, $n-k$ on the other'', the second means ``1 point on the one hand, $n-k-1$ on the other'', etc.
%Now if $n \neq 2k, n \neq 2k + 1, n \neq 2k-1$, then the configurations ``$k$ points on one side'' can be treated exactly in the same way as the configurations where all the points are inside.

%\subsection{Word problem for $\Gamma_{k+1}^k$}
%
%The word problem for the groups $\Gamma_{k+1}^k$ is solvable because they satisfy the small cancellation condition $C(6,3)$. Indeed, the conjugacy problem is solvable too.

\subsection{Triangulations of polyhedra and groups $\Gamma_n^k$}

Triangulations of a given convex polyhedra are well studied. If we deal with stable triangulations, the set of those form a simplicial complex \cite{GKZ}. In the present section we show that they can be treated as elements of a certain group provided that some fixed triangulation plays the role of the unit element.

Let $\bx = \{x_1,\dots,x_n\}\subset \mathbb{R}^d$,  $n>d$, be the vertex set of a polytope $W$ in general position, namely, any $d+1$ points of $\bx$ do not belong to one hyperplane. Let $\mt$ be the set of regular triangulations of $W$ (see page~\pageref{def:regular_triangulation}). We can consider $\mt$ as a graph whose vertices are regular triangulations and whose edges are Pachner moves. This graph is connected, see~\cite{Matveev}.

Choose a regular triangulation $T_0\in\mt$. For any other triangulation $T\in\mt$ there is a path $\gamma=e_1e_2\dots e_l$ from $T_0$ to $T$ where each edge $e_i$ is a Pachner move of type $(P_i,Q_i)$, see page~\pageref{def:type_Pacher_move}. We assign the word $\varphi(T)=\prod_{i=1}^l a_{P_i,Q_i}\in \Gamma_n^{d+2}$ to the triangulation $T$.

Let $Cay(\Gamma_n^{d+2})$ be the Cayley graph of the group $\Gamma_n^{d+2}$ corresponding to the group presentation given in \cref{def:Gamma_nk}.

\begin{theorem}\label{thm:Gamma_and_polytope_triangulations}
The correspondence $T\mapsto \varphi(T)$ defines an embedding of the graph $\mt$ into the graph $Cay(\Gamma_n^{d+2})$.
\end{theorem}

\begin{proof}
We must show first that the map $\varphi$ is well-defined, i.e. that the element $\varphi(T)$ does not depend on the choice of a path connecting $T_0$ with $T$.

It follows from the theory of I.M. Gelfand, M.M. Kapranov and A.V. Zelevinsky~\cite[Chapter 7]{GKZ} that the graph $\mt$ is the edge graph of some polytope $\Sigma(W)$ (called the {\em secondary polytope} of $W$), see \cref{fig:secondary_polytope}. The faces of $\Sigma(W)$ correspond to triangulations which coincide everywhere except for divisions of two subpolytopes with $d+2$ vertices  or a subpolytope with $d+3$ vertices in $W$. Hence, the faces give exactly the relations in $\Gamma_n^{d+2}$ which appear in \cref{def:Gamma_nk}. Thus, two different paths $\gamma$ and $\gamma'$, connecting $T_0$ with some triangulation $T$, produce words $\varphi(\gamma)$ and $\varphi(\gamma')$ which differ by relations in $\Gamma_n^{d+2}$, so $\varphi(\gamma)=\varphi(\gamma')\in \Gamma_n^{d+2}$.

\begin{figure}
\centering\includegraphics[width=240pt]{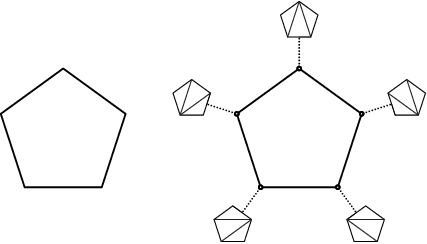}
\caption{A pentagon and its secondary polytope.}\label{fig:secondary_polytope}
\end{figure}

In order to proof that $\varphi$ is injective we construct a left inverse map to it. A {\em formal $l$-dimensional simplex} is any subset $P\subset \{1,\dots,n\}$, $|P|=l+1$. Let $C_l(n)$ be the linear space over $\mathbb{Z}_2$ whose basis consists of all the $l$-dimensional formal simplices. We can identify $C_l(n)\simeq \mathbb{Z}_2^{n\choose{l+1}}$ with the space of $l$-chains of the simplicial complex for a $(n-1)$-dimensional simplex whose vertices are marked with numbers $1,2,\dots,n$. The differential $\partial\colon C_l(n)\to C_{l-1}(n)$ of this complex is defined by the formula
$
\partial(P)=\sum_{p\in P}P\setminus\{p\}.
$

Any triangulation $T$ of $W$ defines a $d$-chain $c(T)=\sum_{P\colon \bx_P\in T}P$ such that $\partial c(T) = \partial W$ where $\partial W$ is the formal sum of $(d-1)$-dimensional faces forming the boundary of the polytope $W$. The triangulation $T$ can be restored from its chain $c(T)$ as the support of the chain. Thus, we can identify triangulations of $W$ with a subset of $C_d(n)$.

 If a triangulation $T'$ differs form a triangulation $T$ by applying a Pachner move of type $(P,Q)$ then $c(T')-c(T)=\partial(P\cup Q)$. Define a homomorphism $\psi\colon\Gamma_n^{d+2}\to C_{d}(n)$ by the formula
$$\psi(a_{P,Q})= \partial(P\cup Q) = \sum_{i\in P\cup Q}(P\cup Q)\setminus\{i\}.$$

The map $\psi$ annihilates the relations of $\Gamma_n^{d+2}$ because they come from cycles of triangulations of some polytope with $d+3$ vertices (or two polytopes with $d+2$ vertices). The image of a relation is equal to the difference of the chains for the final triangulation and the initial one. But these triangulations coincide so the image is zero. Thus, $\psi$ defines a correct homomorphism from $\Gamma_n^{d+2}$ to $C_{d}(n)\simeq\mathbb{Z}_2^{n\choose{d+1}}$.

For any regular triangulation $T$ consider a path $\gamma=e_1e_2\dots e_l$ which connects $T_0$ with $T$ in $\mt$. Any edge $e_i$ of $\gamma$ corresponds to a Pachner move of type $(P_i,Q_i)$. Then
$$
c(T)-c(T_0)=\sum_{i=1}^l\partial(P_i\cup Q_i) = \sum_{i=1}^l\psi(a_{P_i,Q_i})=\psi(\varphi(T)).
$$
Thus, the element $\varphi(T)$ determines the chain $c(T)$, hence, it uniquely determines the triangulation $T$.
\end{proof}

\begin{example}
Let $W\subset\mathbb{R}^2$ be a convex $n$-gon. Its secondary polytope $\Sigma(W)$ is the associahedron (Stasheff polytope) $K_{n-1}$~\cite{Stasheff}, see \cref{fig:associahedron}. Thus, we have the following statement:

\begin{figure}
\sidecaption{}
\centering\includegraphics[width=80pt]{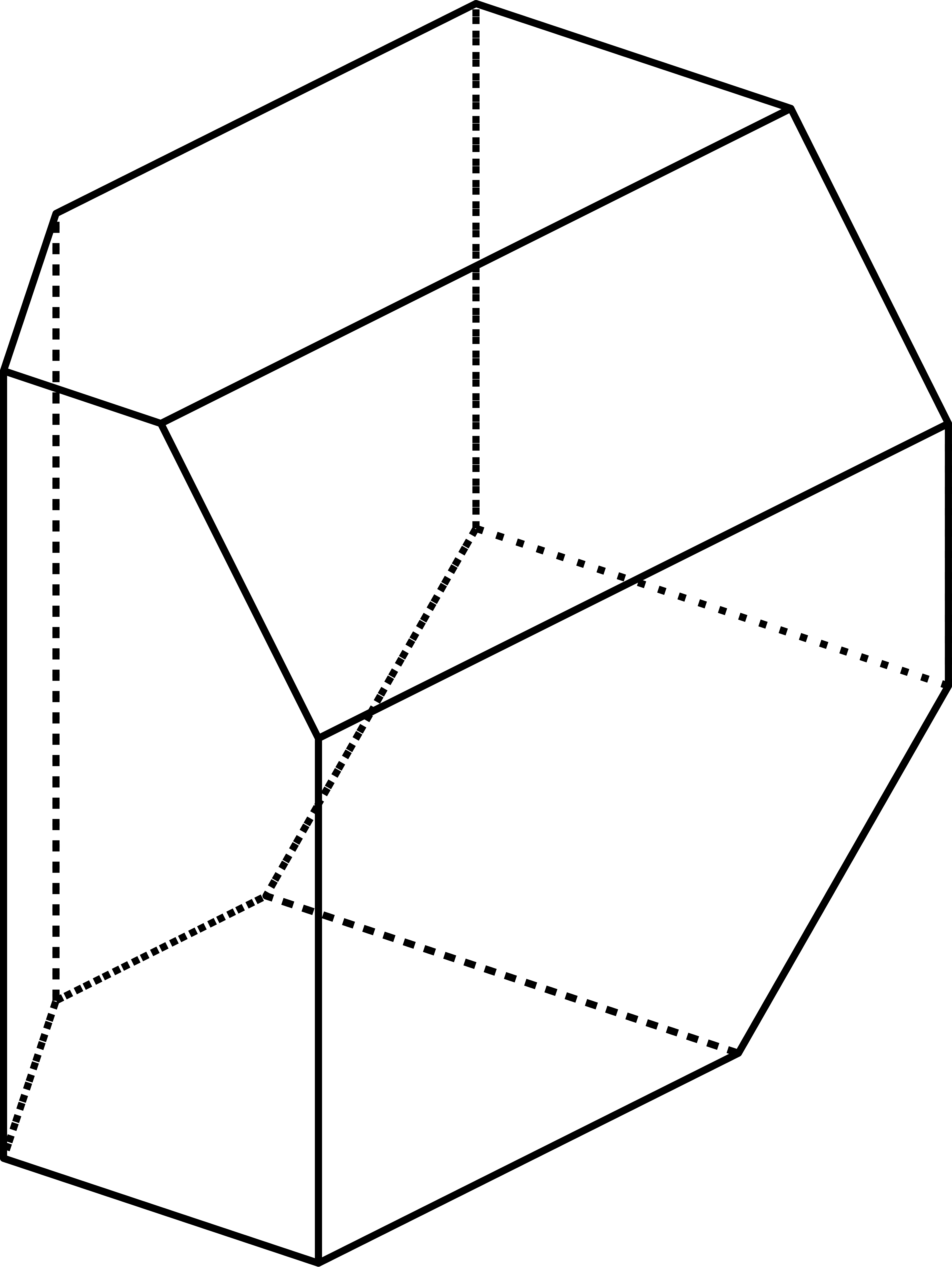}
\caption{Associahedron $K_5$.}\label{fig:associahedron}
\end{figure}

\begin{corollary}
The graph $Cay(\Gamma_n^{4})$ contains a subgraph isomorphic to the edge graph of the associahedron $K_{n-1}$.
\end{corollary}

\end{example}

%\subsection{Representations of the groups $\Gamma_{k+2}^k$}
%
%Here we consider an example of representation for the group $\Gamma_{k+2}^k$ into $\Z_2\oplus\Z_2$.
%
%For each unordered pair $(ij)$ of $k+2$ elements we assign the letter $a$ or the letter $b$ or $1$ in such a way that every row $(jk)$ with a fixed $k$ contains even number of the letters $a$ and even number of the letter $b$.
%
%For example, when $k=5$
%\begin{eqnarray*}
%(12), (23), (34), (41) \mapsto a,\\
%(56), (67), (75) \mapsto b,\\
%\mbox{the other pairs} \mapsto 1.
%\end{eqnarray*}
%
%Then we can assign $a$, $b$ or $1$ to each generator (it contains $k$ indices of $k+2$) of the group $\Gamma_{k+2}^k$. Thus, we obtain a homomorphism into $\Z_{2}\oplus \Z_{2}$.

\section{Other developments: the groups $\tilde{\Gamma}_n^k$}

In the present section we provide constructions resembling the $\Gamma_n^k$ groups but different in some aspects, which prove useful in certain situations where the $\Gamma_n^k$ groups are not sufficient.
%\subsection{The groups $\tilde{\Gamma}_n^k$}
%We consider the following slight variation on the groups $\Gamma_n^k$. 
Geometrically speaking, here we consider oriented triangulations. Therefore, the indices of the generators of the groups are not independent and do not freely commute as was seen, for example, in the groups $\Gamma_n^5$ (see \cref{def:Gamma_n5}, first relation).

To be precise, we introduce the following:

\begin{definition}\label{def:tildeGamma_nk}
Let $5\le k\le n$. The group $\tilde{\Gamma}_n^k$ is the group with the generators
$$
\{ a_{P,Q}\,|\,P,Q \; \hbox{--- cyclically ordered subsets of} \; \{1,\dots,n\}, P\cap Q=\emptyset, |P\cup Q|=k, |P|,|Q|\ge 2\},
$$
and the following relations:
\begin{enumerate}
\item[1,2,3.] the relations 1, 2, 3 from \cref{def:Gamma_nk};
\item[4.] $a_{Q,P}=a_{Q',P'}$, where $Q=Q', P=P'$ as unordered sets, and as cyclically ordered sets $Q$ differs from $Q'$ by one transposition and $P$ differs from $P'$ by one transposition.
%\item $a_{Q,P}=a_{P,Q}^{-1}$;
%\item {\em far commutativity}: $a_{P,Q}a_{P',Q'}=a_{P',Q'}a_{P,Q}$ for each generators $a_{P,Q}$, $a_{P',Q'}$ such that
%\begin{align*}
%|P\cap(P'\cup Q')|&<|P|, & |Q\cap(P'\cup Q')|&<|Q|,\\
% |P'\cap(P\cup Q)|&<|P'|, & |Q'\cap(P\cup Q)|&<|Q'|;
%\end{align*}
%\item {\em $(k+1)$-gon relations}: for any standard Gale diagram $\bar Y$ of order $k+1$ and any subset $M=\{m_1,\dots,m_{k+1}\}\subset \bar{n}$
%$$ \prod_{i=1}^{k+1} a_{M_R(\bar Y,i), M_L(\bar Y,i)}=1, $$
%where $M_R(\bar Y,i) = \{m_j\}_{j\in R_{\bar Y}(i)}$,  $M_L(\bar Y,i) = \{m_j\}_{j\in L_{\bar Y}(i)}$.
\end{enumerate}
\end{definition}

\begin{figure}%{R}{0.45\textwidth}
\sidecaption{}
\includegraphics[width = 0.5\textwidth]{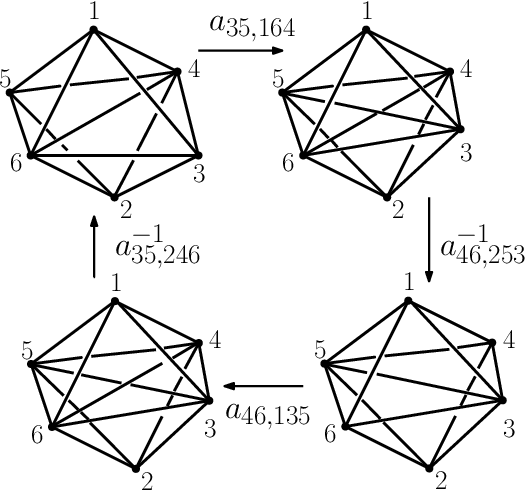}
\caption{A movement of a point around the configuration of four points on one circle}\label{exa_4word}
\end{figure}

These groups have the same connection to geometry and dynamics as the groups $\Gamma_n^k$ defined above. To illustrate that, consider the following dynamical system.
\begin{example}
	Let us have a dynamical system describing a movement of a point around the configuration of four points on one circle, see \cref{exa_4word}. Such system may be presented as the word in the group $\tilde{\Gamma}_6^5$ $$w=a_{35, 164} a^{-1}_{46, 253} a_{46, 135} a^{-1}_{35, 246}.$$

	We can show that $\phi(\alpha)$ is nontrivial in the abelianisation $(\tilde{\Gamma}_6^5)_{ab}=\tilde{\Gamma}_6^5/[\tilde{\Gamma}_6^5,\tilde{\Gamma}_6^5]$ of the group $\tilde{\Gamma}_6^5$. Computer calculations show that the group $(\tilde{\Gamma}_6^5)_{ab}$ can be presented as the factor of a free commutative group with $120$ generators modulo $2\cdot 6!=1440$ relations that span a space of rank $90$ if we work over $\mathbb{Z}_2$. Adding the element $w$ to the relations increases the rank to $91$, so the element $w$ is nontrivial in $(\tilde{\Gamma}_6^5)_{ab}$, therefore it is nontrivial in $\tilde{\Gamma}_6^5$.

	Hence, we have encountered a peculiar new effect in the behaviour of $\tilde{\Gamma}_{n}^{5}$ which is not the case of $G_{n}^{k}$. Certainly, the abelianisation of $G_{n}^{k}$ is non-trivial and very easy to calculate since any generator enters each relation evenly many times. However, this happens not only for relations but for any words which come from braids. Thus, any abelianisations are trivial in interesting cases. For $\tilde{\Gamma}_{n}^{k}$, it is an interesting new phenomenon, and the invariants we have demonstrated so far are just the tip of the iceberg to be investigated further.
	
	Note that both for the groups $\Gamma_n^k$ and $\tilde{\Gamma}_n^k$ we have many invariants since the corank of those groups is big.
\end{example}

\end{document}